\documentclass[11pt]{article}
\usepackage[a4paper,margin=32mm]{geometry}
\usepackage{amsmath,amssymb,amsthm,mathtools}
\usepackage{microtype}
\usepackage[hidelinks]{hyperref}
\usepackage{enumitem}
\usepackage{booktabs}

\newtheorem{theorem}{Theorem}[section]
\newtheorem{proposition}[theorem]{Proposition}
\newtheorem{corollary}[theorem]{Corollary}
\newtheorem{lemma}[theorem]{Lemma}
\newtheorem{remark}[theorem]{Remark}
\newtheorem{example}[theorem]{Example}

\newcommand{\cO}{\mathcal O}
\newcommand{\PP}{\mathbb P}
\newcommand{\CC}{\mathbb C}
\newcommand{\Tan}{\operatorname{Tan}}
\newcommand{\Sec}{\operatorname{Sec}}
\newcommand{\Gr}{\operatorname{Gr}}
\newcommand{\tdeg}{\tau_{\mathrm{tan}}}

\title{Chern bounds and tangent geometry \\ of polarized Calabi--Yau threefolds}
\author{Atsushi Kanazawa}
\date{}

\begin{document}
\maketitle

\begin{abstract}
We study the numerical geography and tangent geometry of very amply polarized Calabi--Yau threefolds $(X,H)$ through the positivity of the first jet bundle $J^1(H)$. 
Let $d=\int_X H^3$, $c=\int_Xc_2(X) H$, and $e=\int_X c_3(X)$.  
We obtain $e\ge-5d-c-c^2/(4d)$ by mixed intersections on $\PP(J^1(H)^*)$, and $e\ge40d\left[(1-\frac{c}{10d})^{3/2}-1\right]$ by a volume estimate for a perturbed tautological class. 
In particular, the latter inequality shows $e+6c\ge0$, improving Sun's inequality $e+10c\ge0$. 
As consequences, we show (1) the uniform Hodge bounds $-4d-80\le h^{1,1}(X)-h^{2,1}(X)\le173d/66$, 
(2) the lower bound $\deg X^\vee\ge78$ for the dual hypersurface, and (3) the upper bound $d\le34$ in the critical case $X\subset\PP^6$. 
\end{abstract}

\tableofcontents


\section{Introduction}

We study the numerical geography and tangent geometry of polarized Calabi--Yau threefolds through the positivity of the first jet bundle, which simultaneously constrains the Chern and Hodge numbers and governs the tangent and dual varieties.
 
Throughout, we work over $\CC$.
A Calabi--Yau threefold is a smooth projective threefold $X$ with $K_X=0$ and $H^1(X,\cO_X)=0$.
Let $H$ be a very ample divisor\footnote{If $L$ is an ample divisor, Oguiso--Peternell \cite{OP} show that $10L$ is very ample.
The results below therefore apply with $H=10L$, following the same reduction as in \cite{KW}.} whose complete linear system embeds $X$ in $\PP^N$, 
and set $d=\int_X H^3$, $c=\int_X c_2(X) H$, and $e=\int_X c_3(X)$.
Since $e=2(h^{1,1}(X)-h^{2,1}(X))$, every bound for $e$ gives a bound for the Hodge-number difference. 

\vspace{1mm}

\noindent
{\bf The first jet bundle.}
The source of positivity we use is the first jet bundle $J^1(H)$, which plays a central role in both tangent geometry and projective duality. 
On the one hand, its Segre class describes the tangent geometry via the classical factorization of the tangential Chern gap
\begin{equation}\label{eq:intro-factor}
 \omega(X,H)=20d-6c-e=\tdeg(X)\deg\Tan(X),\qquad N\ge6,
\end{equation}
where $\Tan(X)$ is the tangent variety and $\tdeg(X)$ is the degree of the tangent-incidence morphism $\PP(J^1(H)^*)\to\Tan(X)\subset\PP^N$. 
On the other hand, its top Chern class computes the degree of the projective dual, $\deg X^\vee=4d+2c-e$. 
The inequality
\begin{equation}\label{eq:intro-classical}
 e\le20d-6c
\end{equation}
is classical \cite{LS, CDG, LZ}. 
A key observation of this paper is that the projective bundle $\mathcal T=\PP(J^1(H)^*)$ has two more positivity properties, which give rise to two lower bounds for $e$. 

\vspace{1mm}

\noindent
{\bf Mixed intersections.}
The tautological class $\xi$ and the pullback $\eta$ of $H$ are nef on the sixfold $\mathcal T$, and the three relevant intersection numbers are given explicitly by our basic invariants:
\[
 \int_{\mathcal T}\xi^4\eta^2=4d,\qquad
 \int_{\mathcal T}\xi^5\eta=10d-c,\qquad
 \int_{\mathcal T}\xi^6=\omega(X,H). 
\]
Therefore, by the Khovanskii--Teissier inequality \cite[Section~1.6]{Laz}, we obtain (Theorem~\ref{thm:logconcave})
\begin{equation}\label{eq:intro-nonlinear}
 e\ge-5d-c-\frac{c^2}{4d}. 
\end{equation}
This improves the Schur-class bound $e\ge-20d-2c$ and the equality is attained by the quintic threefold. 

\vspace{1mm}

\noindent
{\bf Volume estimate.}
The positivity of $J^1(H)$ also implies $0\le c\le10d$, and a volume estimate for the perturbed tautological class $\xi-t\eta$ provides our second bound (Theorem~\ref{thm:jet-volume})
\begin{equation}\label{eq:intro-volume}
 e \ge 40d \Big[\left(1-\frac{c}{10d}\right)^{3/2}-1\Big]. 
\end{equation}
In particular, $e+6c\ge0$, improving Sun's inequality $e+10c\ge0$ \cite[Corollary~6.3]{Sun}.
The equality $e+6c=0$ holds if and only if $c=0$, i.e. $X$ admits a finite \'{e}tale cover by an abelian threefold (type A in \cite{OS, HK}).
Equality in \eqref{eq:intro-volume} is attained both by threefolds of type A and by the quintic threefold.
The proof combines the semistability of the symmetric powers $\operatorname{Sym}^mJ^1(H)$, Langer's section estimate \cite[Theorem~3.3]{Langer}, and the algebraic Morse inequality \cite[Proposition~3.3]{BFJ}. 

By \eqref{eq:intro-factor}, the linear consequence $e+6c\ge0$ can be interpreted geometrically as follows:
\begin{equation}\label{eq:intro-tangent-degree}
 \tdeg(X)\deg\Tan(X)\le20\deg X,\qquad N\ge6. 
\end{equation}

\vspace{1mm}

\noindent
{\bf Hodge and degree bounds.}
Combining \eqref{eq:intro-nonlinear} with hyperplane-section and tangent-variety estimates, we obtain (Theorem~\ref{thm:Hodge})
\[
 -4d-80\le h^{1,1}(X)-h^{2,1}(X)\le\frac{173}{66}d
\]
substantially improving the previously known bounds $-36d-80\le h^{1,1}(X)-h^{2,1}(X)\le6d+40$ obtained in \cite{KW}. 
The lower bound is sharp and equality is attained only by the quintic threefold, and for $d>5$ we obtain the stronger lower bound $\left\lceil-4d-40-\frac{200}{d}\right\rceil \le h^{1,1}(X)-h^{2,1}(X)$. 
(As for the upper bound, see Remark~\ref{rem:Hodge-comparison} for the equality case.) 
These bounds are of relevance because bounding the topology of Calabi--Yau threefolds is a crucial problem in their classification, and bounding the Euler characteristic $e$ would be a first step. 
Since the degree of the projective dual is the top Chern class of $J^1(H)$ (Theorem~\ref{thm:dual-hypersurface}), the same estimates bound it from below uniformly in the polarization (Theorem~\ref{thm:dual-lower}): 
\[
\deg X^\vee=4d+2c-e \ge78.
\]

\vspace{1mm}

\noindent
{\bf Ambient dimension.}
The ambient dimension of $X \subset \PP^N$ determines both the numerical constraints and the geometry of the tangent variety. 
We distinguish the low-codimension range $N=4,5$, the critical case $N=6$, and the stable range $N\ge7$. 
For $N=4,5$, the possibilities are the quintic threefold $X_5\subset\PP^4$ and the complete intersections $X_{2,4},X_{3,3}\subset\PP^5$, and the equality holds in \eqref{eq:intro-classical}. 
In the critical case $N=6$, by Tonoli's identities $c=84-2d$ and $e=49d-588-d^2$ in \cite{Tonoli}, the tangent degree is a quadratic polynomial in $d$ (Theorem~\ref{thm:P6exact}):
\[
 \tdeg(X)=d^2-17d+84. 
\]
Combined with the results of Kapustka--Kapustka \cite{KK}, this implies $\tdeg(X)\ge24$, with equality if and only if $X=X_{2,2,3}\subset\PP^6$. 
Here the new inequality $e+6c\ge0$ becomes $-d^2+37d-84\ge0$, and we obtain (Theorem~\ref{thm:P6window})
\[
 12\le d\le34, 
\]
improving the previously known upper bound $d\le41$ \cite[Corollary~2.2]{KK}. 
In the stable range $N\ge7$, from results on secant-defective threefolds and tangent degrees, combined with the nonrationality of $X$, we obtain (Theorem~\ref{thm:stable})
\begin{equation}\label{eq:intro-stable}
 e\le20d-6c-2(N-5)-2. 
\end{equation}

\vspace{1mm}

\noindent{\bf Quadratic equations.}
Finally, suppose that the homogeneous ideal of $X\subset\PP^N$ is generated by quadrics.
Then the twisted conormal bundle $N_{X/\PP^N}^*(2H)$ is globally generated, and the Khovanskii--Teissier inequality applies to it as well.
As a consequence, such an $X\subset\PP^7$ is the complete intersection $X_{2,2,2,2}$, while such an $X\subset\PP^8$ has $(d,c)=(20,68)$ and $e\in\{-120,-118\}$ (Proposition~\ref{prop:conormal}, Corollary~\ref{cor:quadratic-P8}).
The value $e=-120$ is realized by the intersection of a general hyperplane section of $\Gr(2,5)$ with two quadrics (Example~\ref{ex:quadratic-P8}).



\section{The first jet bundle and tangent geometry} 
\label{sec:background}

The first jet bundle $J^1(H)$ of a very ample divisor is the main tool of this paper.
It is globally generated since $|H|$ separates first jets, and its Segre classes describe the tangent geometry. 
In this section, we fix the conventions, derive the tangent-incidence formula, and introduce the tangential Chern gap, the numerical invariant which measures this geometry.


\subsection{Notation and conventions}\label{subsec:notation}

Let $X\subset\PP^N=\PP(V)$ be a smooth projective threefold, where $\PP(V)$ denotes the space of lines in the vector space $V$.
We call $X$ nondegenerate if it is not contained in a hyperplane. 
``General'' means on a dense Zariski-open subset of the variety in question.
We write $T_xX\subset\PP(V)$ for the embedded projective tangent space at $x\in X$.

For a vector bundle $E$ on $X$, we use the convention that $\PP(E)$ parametrizes lines in $E$,
with $\cO_{\PP(E)}(-1)\subset p^*E$ the tautological subbundle and $p:\PP(E)\rightarrow X$ the projection.
Then the Segre classes are given by
\begin{equation}\label{eq:segre-convention}
 s_i(E):=p_*\Bigl(c_1\bigl(\cO_{\PP(E)}(1)\bigr)^{\operatorname{rk}E-1+i}\Bigr)
 \qquad(i\ge0),
\end{equation}
and satisfy $s(E)=c(E)^{-1}$ and $s_i(E^*)=(-1)^is_i(E)$. 

In intersection formulas, we use the same symbol for a Cartier divisor $D$ and its class $c_1(\cO_X(D))$, and we write $c_i(X)=c_i(T_X)$.
For the fixed polarized threefold $(X,H)$, we write
\[
 d:=\int_X H^3,\qquad
 c:=\int_X c_2(X)H,\qquad
 e:=\int_X c_3(X).
\]


\subsection{The first jet bundle and tangent incidence}\label{subsec:jet-incidence}

We assume from now on that $X$ is nondegenerate. 
Let $H$ be a hyperplane divisor, so that $\cO_X(H)\cong\cO_{\PP^N}(1)|_X$, and write $h=c_1(\cO_{\PP^N}(1))$.
The first jet bundle $J^1(H)$ has rank $4$: its fibre at $x\in X$ is the vector space $\cO_X(H)_x/\mathfrak m_x^2\cO_X(H)_x$, where $\mathfrak m_x$ is the maximal ideal
of $\cO_{X,x}$.
It fits into the jet sequence
\[
 0\longrightarrow\Omega_X^1(H)
 \longrightarrow J^1(H)
 \longrightarrow\cO_X(H)\longrightarrow0.
\]
A linear system separates first jets if its evaluation map to this fibre is surjective at every $x$.
Since $X\subset\PP(V)$ is an embedding, the defining linear system $V^*\subseteq H^0(X,\cO_X(H))$ separates first jets, so its evaluation $V^*\otimes\cO_X\to J^1(H)$ is surjective.\footnote{
Although $H$ is very ample, we do not assume the defining linear system to be complete in Sections~\ref{sec:background}--\ref{sec:dual}.}
Combining the conormal sequence with the Euler sequence identifies its kernel with $N_{X/\PP^N}^*(H)$.
Dualizing, we obtain
\begin{equation}\label{eq:jet-dual-sequence}
 0\longrightarrow J^1(H)^*
 \longrightarrow V\otimes\cO_X
 \longrightarrow N_{X/\PP^N}(-H)\longrightarrow0.
\end{equation}
Therefore, we have
\begin{equation}\label{eq:normal-segre}
 c\bigl(N_{X/\PP^N}(-H)\bigr)=c(J^1(H)^*)^{-1}=s(J^1(H)^*).
\end{equation}
See for example \cite{BS,DRS,LZ}.

We define $\mathcal T:=\PP(J^1(H)^*)$ and let $p:\mathcal T\to X$ be the projection.
By \eqref{eq:jet-dual-sequence}, the projectivization of the fibre of $J^1(H)^*$ at $x$ is the embedded tangent space $T_xX$.
Therefore $\mathcal T$ parametrizes pairs $(x,z)$ such that $z\in T_xX$, and the tangent-incidence morphism is $q:\mathcal T\to\PP^N$, $(x,z)\mapsto z$.
By the inclusion in \eqref{eq:jet-dual-sequence}, we have $\cO_{\mathcal T}(-1)\subset q^*\cO_{\PP^N}(-1)$, and hence
\[
 \xi:=c_1\bigl(\cO_{\mathcal T}(1)\bigr)=q^*h.
\]
Since \(\operatorname{rk}J^1(H)=4\), the projective bundle $\mathcal T$ is a smooth sixfold,
although its image, the tangent variety $\Tan(X):=q(\mathcal T)$, may have smaller dimension.
As $\mathcal T$ is proper, $\Tan(X)$ is closed.
On the other hand, we write $\Sec(X)$ for the Zariski closure of the union of secant lines to $X$.
A threefold $X$ is called secant defective if $\dim\Sec(X)<\min\{7,N\}$, the expected dimension of its secant variety.


\subsection{Tangent degree}\label{subsec:tangent-degree-background}

By \eqref{eq:normal-segre} and the definition \eqref{eq:segre-convention} of the Segre classes, applied to $J^1(H)^*$, we have
\begin{equation}\label{eq:classical-tangent}
 \int_Xc_3\bigl(N_{X/\PP^N}(-H)\bigr)
 =\int_Xs_3(J^1(H)^*)
 =\int_{\mathcal T}\xi^6.
\end{equation}
When $q:\mathcal T\to\Tan(X)$ is generically finite, we define the tangent degree by
\[
 \tdeg(X):=\deg\bigl(q:\mathcal T\to\Tan(X)\bigr).
\]
It counts the points $x\in X$ such that $z\in T_xX$ for a general $z\in\Tan(X)$.
Therefore we have
\begin{equation}\label{eq:tandegree}
 \int_{\mathcal T}\xi^6
 =
 \begin{cases}
 \tdeg(X)\deg\Tan(X),&\dim\Tan(X)=6,\\
 0,&\dim\Tan(X)<6.
 \end{cases}
\end{equation}
For the complete embedding defined by a very ample divisor $H$, we also write $\Tan(X,H)$ and $\tdeg(X,H)$ to indicate the polarization.
We say that the embedding has tangent birationality if its tangent-incidence morphism is birational, i.e. $\tdeg(X)=1$.
The analogous factorization in arbitrary dimension is proved in \cite[Lemma~3.2]{CDG}.

\begin{remark}\label{rem:singular-tangent-degree}
For an irreducible, possibly singular variety, we use the closure of the tangent incidence over its smooth locus, and the tangent degree is defined when the projection to the tangent variety is generically finite.
\end{remark}

\begin{lemma}\label{lem:lowcodim}
For a smooth nondegenerate threefold $X\subset\PP^N$, the following are equivalent:
\[
 \int_Xc_3\bigl(N_{X/\PP^N}(-H)\bigr)=0,\qquad
 \dim\Tan(X)<6,\qquad
 N\le5.
\]
\end{lemma}

\begin{proof}
The equivalence of the first two conditions follows from \eqref{eq:tandegree}.
The implication $N\le5\Rightarrow\dim\Tan(X)<6$ is immediate.
Conversely, if $N\ge6$ and $\dim\Tan(X)<6$, then $\Tan(X)=\Sec(X)$ by Fulton--Hansen, so that $\dim\Sec(X)<6\le N$.
Zak's bound on the secant defect then forces $\dim\Sec(X)\ge\frac32\dim X+1=\frac{11}2$, hence $\dim\Sec(X)\ge6$, a contradiction (see \cite{FH} and \cite[Chapter~II]{Zak}). 
\end{proof}


\subsection{\texorpdfstring{$K$}{K}-trivial threefolds and the tangential Chern gap}\label{subsec:tangential-gap}
Let us further assume that $K_X=0$. 
Since $c_1(X)=0$, by the jet sequence we have
\begin{align}\label{eq:jetchern}
 c_1(J^1(H))&=4H,\notag\\
 c_2(J^1(H))&=c_2(X)+6H^2,\\
 c_3(J^1(H))&=-c_3(X)+2c_2(X)H +4H^3.\notag
\end{align}
Therefore we have
\begin{equation}\label{eq:CYsegre}
 s_3(J^1(H)^*)=20H^3-6c_2(X)H -c_3(X),
\end{equation}
and hence, by Lemma~\ref{lem:lowcodim},
\begin{equation}\label{eq:CYfactor}
 20d-6c-e
 =\int_{\mathcal T}\xi^6
 =
 \begin{cases}
 \tdeg(X)\deg\Tan(X),&N\ge6,\\
 0,&N\le5.
 \end{cases}
\end{equation}
We call 
\[
\omega(X,H):=20d-6c-e 
\]
the tangential Chern gap, and write $\omega$ when the polarization is clear.
Since $\xi=q^*h$ is nef, $\omega(X,H)\ge0$, and hence $e\le20d-6c$.
This inequality is classical; see Livorni--Sommese \cite[Theorem~(0.7.3)]{LS} and Li--Zheng \cite[Corollary~2.4]{LZ}.

\begin{example}\label{ex:ci-tangent}
Here and below, $X_{d_1,\ldots,d_r}$ denotes a complete intersection of hypersurfaces of degrees $d_1,\ldots,d_r$ with $d_i\ge2$.  
For $X=X_{d_1,\ldots,d_r}\subset\PP^{r+3}$, put $a_i=d_i-1$ and $d=\prod_i d_i$.
By the splitting $N_{X/\PP^{r+3}}(-H)\cong\bigoplus_i\cO_X(a_iH)$, we have
\[
 \int_Xs_3(J^1(H)^*)=d\sum_{i<j<k}a_i a_j a_k.
\]
For a Calabi--Yau complete intersection, \eqref{eq:CYfactor} identifies this number with $\omega(X,H)$.
It vanishes for $r\le2$, equals $24$ for $X_{2,2,3}$, and equals $64$ for $X_{2,2,2,2}$.
Therefore $\tdeg(X_{2,2,3})=24$ since its tangent variety is $\PP^6$, while for $X_{2,2,2,2}$ we obtain the product $\tdeg(X)\deg\Tan(X)=64$.
We compare these values with the dual degrees in Example~\ref{ex:ci-dual}.
\end{example}


\section{Tangent degree in the critical case and the stable range}
\label{sec:tangentdegree}

By Lemma~\ref{lem:lowcodim}, the tangent-incidence morphism is generically finite if and only if $N\ge6$, so that the tangential Chern gap reflects the tangent geometry only for $N\ge6$.
We will see that the two remaining cases are of different nature.
In the critical case $N=6$ the tangent variety fills the ambient space, while in the stable range $N\ge7$ it is a proper subvariety of the secant variety, and a lower bound for its degree refines the classical inequality $e\le20d-6c$.


\subsection{The critical case $N=6$}\label{subsec:critical-range}

Recall that a Calabi--Yau threefold in $\PP^4$ is a quintic threefold $X_5$ and that those in $\PP^5$ are the complete intersections $X_{2,4}$ and $X_{3,3}$.
The case $N=6$, which we call the critical case, is the first case beyond this classification, and the classification of nondegenerate Calabi--Yau threefolds in $\PP^6$ is still open.
However, their Chern numbers and tangent degree are determined by the projective degree alone, as follows.

\begin{theorem}\label{thm:P6exact}
Let $X\subset\PP^6$ be a nondegenerate Calabi--Yau threefold with hyperplane divisor $H$.
Then the Chern numbers and tangent degree are determined by $d$:
\begin{align}
c&=84-2d,\label{eq:P6c2}\\
 e&=49d-588-d^2,\label{eq:P6c3}\\
 \tdeg(X)&=d^2-17d+84
          =(d-8)(d-9)+12.\label{eq:P6tau}
\end{align}
Moreover,
\[
 d\ge12,\qquad \Tan(X)=\PP^6,\qquad \tdeg(X)\ge24,
\]
and equality $\tdeg(X)=24$ holds if and only if $X$ is a complete intersection $X_{2,2,3} \subset\PP^6$.
\end{theorem}

\begin{proof}
By \cite[Proposition~2.1]{KK}, every Calabi--Yau threefold in $\PP^6$ is linearly normal, and hence the restriction map $H^0(\PP^6,\cO_{\PP^6}(1))\to H^0(X,\cO_X(H))$ is an isomorphism as $X$ is nondegenerate.
Then we have $h^0(X,\cO_X(H))=7$, and by Riemann--Roch and Kodaira vanishing
\[
 7=\chi(X,\cO_X(H))=\frac d6+\frac{c}{12},
\]
which is nothing but \eqref{eq:P6c2}.

We next compute $e$.
By the tangent sequence, we have $c(N_{X/\PP^6})=(1+H)^7/c(T_X)$, and taking the degree-$3$ part with $c_1(X)=0$, we may write
\[
 d^2=\int_{\PP^6}[X]^2
 =\int_Xc_3(N_{X/\PP^6})=35d-7c-e
\]
by the self-intersection formula.
Substituting \eqref{eq:P6c2}, we obtain \eqref{eq:P6c3}.

Finally, $\dim\Tan(X)=6$ by Lemma~\ref{lem:lowcodim}, so that $\Tan(X)=\PP^6$ and $\deg\Tan(X)=1$.
Then, by \eqref{eq:CYfactor}, \eqref{eq:P6c2} and \eqref{eq:P6c3}, we have
\[
 \tdeg(X)=20d-6c-e=d^2-17d+84,
\]
which is \eqref{eq:P6tau}.
Kapustka--Kapustka prove that $d\ge12$ and that the degree-$12$ case is exactly the complete intersection $X_{2,2,3}$ (see \cite[Corollary~5.1]{KK}).
Since $d^2-17d+84$ is strictly increasing for $d\ge12$, we obtain $\tdeg(X)\ge24$, where the equality holds exactly for $X_{2,2,3}$.
\end{proof}

\begin{remark}\label{rem:P6class}
The identities for $c$ and $e$ were already obtained by Tonoli \cite{Tonoli}, and in particular we have
\[
 h^{1,1}(X)-h^{2,1}(X)=\frac{49d-588-d^2}{2}.
\]
Kapustka--Kapustka \cite{KK} classify Calabi--Yau threefolds in $\PP^6$ of degree at most $14$, and Tonoli \cite{Tonoli} constructs smooth examples in every degree from $12$ to $17$, while the general classification is still open.
For these degrees, Theorem~\ref{thm:P6exact} reads as follows.
\[
\begin{array}{c|rrrrrr}
d&12&13&14&15&16&17\\ \hline
\tdeg(X)&24&32&42&54&68&84
\end{array}
\]
\end{remark}


\subsection{The stable range $N\ge7$}\label{subsec:stable-range}

Let us next consider the case $N\ge7$, where the tangent variety provides a strict refinement of the classical inequality $e\le20d-6c$.

\begin{theorem}\label{thm:stable}
Let $X\subset\PP^N$ be a smooth nondegenerate projective threefold with $K_X=0$ and $N\ge7$, and let $H$ be its hyperplane divisor.
Then
\[
 \dim\Sec(X)=7,
 \qquad
 \Tan(X)\subsetneq\Sec(X),
\]
and
\begin{equation}\label{eq:tan-lower}
 \deg\Tan(X)\ge2(N-5).
\end{equation}
Hence we have
\begin{equation}\label{eq:stable-omega}
 \omega(X,H)\ge
 \max\{2(N-5)\tdeg(X),\,2(N-5)+2\},
\end{equation}
and therefore
\begin{equation}\label{eq:stable-chern}
 e\le20d-6c-2(N-5)-2.
\end{equation}
\end{theorem}

The following lemma excludes the extremal value $\omega(X,H)=2(N-5)$.

\begin{lemma}\label{lem:minimal-tangent-rational}
Let $Y\subset\PP^N$, $N\ge7$, be an irreducible nondegenerate threefold with $\dim\Tan(Y)=6$ and $\dim\Sec(Y)=7$.
If $\tdeg(Y)\deg\Tan(Y)=2(N-5)$, then $Y$ is rational.
\end{lemma}

\begin{proof}
By \cite[Theorem~27]{HGR}, we have $\deg\Tan(Y)\ge2(N-5)$, and hence the hypothesis forces $\tdeg(Y)=1$ and the equality $\deg\Tan(Y)=2(N-5)$.
Note that the results of \cite{HGR} used below apply to possibly singular varieties, in particular to the projected images of $Y$ (cf. Remark~\ref{rem:singular-tangent-degree}).
Let us project $Y$ successively from general points of $Y$ and of its images until the ambient space becomes $\PP^7$.
Each of these internal projections is birational, since the threefold has codimension at least $2$ at every step (see the proof of \cite[Theorem~4.2]{CR}), and it preserves $\dim\Tan=6$ and $\dim\Sec=7$.
Moreover, by the equality case of \cite[Theorem~27]{HGR} and the projection formula in \cite[Remark~26]{HGR}, it preserves the tangent degree $1$ and lowers the degree of the tangent variety by $2$.
In this way we obtain a threefold $Y'\subset\PP^7$ birational to $Y$ with $\tdeg(Y')=1$ and $\deg\Tan(Y')=4$, where we take $Y'=Y$ if $N=7$.

The tangent variety $\Tan(Y')$ is singular along $Y'$, and its vertex, namely the set of points from which it is a cone, is a proper linear subspace.
Hence a general point $y\in Y'$ lies outside the vertex, and by \cite[Lemma~16]{HGR} the projection $Z\subset\PP^6$ of $Y'$ from $y$ has $\tdeg(Z)=2$.
This projection $Y'\dashrightarrow Z$ is again birational, and $Z$ is birational to its general tangent space $\PP^3$ by \cite[Lemma~11]{HGR}.
Therefore $Y$ is rational.
\end{proof}

\begin{proof}[Proof of Theorem~\ref{thm:stable}]
Recall that every smooth secant-defective threefold is rational by the classification (see \cite{CC} and the summary in \cite{CCR}).
Since $K_X=0$, we have $h^{3,0}(X)=1$, and hence $X$ is not rational and is not secant defective.
Therefore $\dim\Sec(X)=7$, while $\dim\Tan(X)=6$ by Lemma~\ref{lem:lowcodim}, so that $\Tan(X)\subsetneq\Sec(X)$.
Then \eqref{eq:tan-lower} follows from \cite[Theorem~27]{HGR}, and by \eqref{eq:CYfactor} we have $\omega(X,H)\ge2(N-5)\tdeg(X)\ge2(N-5)$.
Here the equality $\omega(X,H)=2(N-5)$ does not occur by Lemma~\ref{lem:minimal-tangent-rational}, since $X$ is not rational.
Finally, $e$ is even by Poincar\'e duality and so is $\omega(X,H)=20d-6c-e$, and hence $\omega(X,H)\ge2(N-5)+2$.
This proves \eqref{eq:stable-omega} and \eqref{eq:stable-chern}.
\end{proof}

\begin{remark} \label{rem:veronese}
The secant nondefectivity in Theorem~\ref{thm:stable} cannot be dropped.
Let $X=\nu_2(\PP^3)\subset\PP^9$ be the image of the quadratic Veronese embedding $\nu_2:\PP^3\hookrightarrow\PP^9$.
Then $\Tan(X)=\Sec(X)$ is the locus of quadratic forms of rank at most $2$, and we have $\tdeg(X)=2$.
In fact, writing $\PP^3=\PP(U)$ with $\dim U=4$, we have $X\subset\PP(\operatorname{Sym}^2U)$ and $T_{[\ell^2]}X=\PP(\ell U)$ for $0\ne\ell\in U$.
For a general form of rank $2$ written as $q=\ell_1\ell_2$, the point $[q]$ lies in $T_{[\ell^2]}X$ if and only if $\ell$ divides $q$, and hence exactly two tangent spaces pass through $[q]$.
In other words, the tangent degree can be greater than $1$ for smooth threefolds with $N\ge7$, but this classical example is secant defective and is excluded in the Calabi--Yau setting by Theorem~\ref{thm:stable}.
\end{remark}

It is worth noting that the estimate \eqref{eq:tan-lower} also provides a numerical criterion for tangent birationality.

\begin{corollary}\label{cor:numerical-tau}
Under the assumptions of Theorem~\ref{thm:stable}, if
\begin{equation}\label{eq:numerical-tau}
 \omega(X,H)<4(N-5),
\end{equation}
then $\tdeg(X)=1$, and hence
\[
 \deg\Tan(X)=\omega(X,H).
\]
\end{corollary}

\begin{proof}
If $\tdeg(X)\ge2$, then by \eqref{eq:tan-lower} and \eqref{eq:CYfactor} we have $\omega(X,H)\ge4(N-5)$, which contradicts \eqref{eq:numerical-tau}.
\end{proof}


\section{Projective duality}\label{sec:dual}

The dual variety provides the second interpretation of the Chern classes of $J^1(H)$.
Recall that the dual variety $X^\vee\subset(\PP^N)^\vee$ is the Zariski closure of the set of hyperplanes tangent to $X$.
For a threefold with trivial canonical class, we will see that $X^\vee$ is automatically a hypersurface and that its degree is computed by the top Chern class of $J^1(H)$.

\begin{theorem}\label{thm:dual-hypersurface}
Let $X\subset\PP^N$ be a smooth nondegenerate projective threefold with $K_X=0$, and let $H$ be its hyperplane divisor.
Then the dual variety $X^\vee\subset(\PP^N)^\vee$ is a hypersurface of even degree.
Moreover,
\begin{equation}\label{eq:dual-degree}
 \deg X^\vee
 =\int_Xc_3(J^1(H))
 =4d+2c-e.
\end{equation}
\end{theorem}

\begin{proof}
By the classification of smooth dual-defective threefolds, a nonlinear dual-defective threefold is a scroll over a curve (see for example \cite{Sierra}).
Suppose that $X\to B$ is such a scroll and let $F\cong\PP^2$ be a general fibre, so that $N_{F/X}\cong\cO_F$.
Then, by the adjunction formula and the triviality of $N_{F/X}$, we have $K_F=K_X|_F=0$, which contradicts $\cO_F(K_F)\cong\cO_{\PP^2}(-3)$, and hence $X$ is not dual defective.
Formula \eqref{eq:dual-degree} is nothing but the classical jet-bundle formula for the degree of the discriminant (see \cite{LPS}) combined with \eqref{eq:jetchern}.
Finally, the topological Euler characteristic $e$ is even by Poincar\'e duality, and hence so is $\deg X^\vee$ by \eqref{eq:dual-degree}.
\end{proof}

In the critical case $N=6$, by Theorem~\ref{thm:P6exact}, we have the explicit dual degree
\[
 \deg X^\vee
 =4d+2(84-2d)-(49d-588-d^2)
 =d^2-49d+756.
\]

Let us summarize.
The first jet bundle determines both the dual and tangent geometries through the two natural codimension-$3$ classes
\[
 \begin{array}{ccc}
 \displaystyle\int_Xc_3(J^1(H))=4d+2c-e
 &\longleftrightarrow&X^\vee,\\[1mm]
 \displaystyle\int_Xs_3(J^1(H)^*)=20d-6c-e
 &\longleftrightarrow&\Tan(X).
 \end{array}
\]
Namely, the integral of $c_3(J^1(H))$ is the degree of the dual hypersurface, while that of $s_3(J^1(H)^*)$ is the product $\tdeg(X)\deg\Tan(X)$ for $N\ge6$.

\begin{example}\label{ex:ci-dual}
Recall that the dual variety of a smooth hypersurface $Y\subset\PP^N$ of degree $d\ge2$ has degree $d(d-1)^{N-1}$.
For the Calabi--Yau complete intersections of Example~\ref{ex:ci-tangent}, the splitting of $N_{X/\PP^{r+3}}(-H)$ and \eqref{eq:normal-segre} show $c(J^1(H)^*)=\prod_{i=1}^r(1+a_iH)^{-1}$, and hence we have
\[
 c(J^1(H))=\prod_{i=1}^r(1-a_iH)^{-1},
 \qquad
 \deg X^\vee=d\sum_{1\le i\le j\le k\le r}a_i a_j a_k.
\]
In other words, the tangent and dual geometries are described by the third elementary and complete homogeneous symmetric polynomials in $a_1,\ldots,a_r$, respectively.
Then we have the following.
\[
\begin{array}{c|c|c|c}
 X&d&\omega(X,H)&\deg X^\vee\\ \hline
 X_5\subset\PP^4&5&0&320\\
 X_{2,4}\subset\PP^5&8&0&320\\
 X_{3,3}\subset\PP^5&9&0&288\\
 X_{2,2,3}\subset\PP^6&12&24&312\\
 X_{2,2,2,2}\subset\PP^7&16&64&320
\end{array}
\]
For the quintic threefold $X_5$, the value $\deg X_5^\vee=320$ agrees with $5(5-1)^3$, while the tangential Chern gap vanishes.
For $X_{2,2,2,2}$, the table shows that $\tdeg(X_{2,2,2,2})\deg\Tan(X)=64$ for every smooth member.
\end{example}


\section{Refined Chern, Hodge, and degree bounds}\label{sec:otherbounds}

Throughout this section, $(X,H)$ is a very amply polarized Calabi--Yau threefold and $X\subset\PP^N$ is the complete embedding defined by $|H|$, except in Subsections~\ref{subsec:jet-mixed} and~\ref{subsec:jet-volume}, where we only assume that $X$ is a smooth projective threefold with $K_X=0$ and do not use $H^1(X,\cO_X)=0$.
By Riemann--Roch and Kodaira vanishing, we have
\begin{equation}\label{eq:RR}
 N+1=h^0(X,\cO_X(H))=\chi(X,\cO_X(H))=\frac d6+\frac c{12},
 \qquad
 c=12(N+1)-2d.
\end{equation}
In particular, $c$ is even, and so is $e=2(h^{1,1}(X)-h^{2,1}(X))$.
We will frequently use \eqref{eq:RR} to replace $c$ by the embedding dimension $N$ and vice versa.

In Subsections~\ref{subsec:jet-mixed} and~\ref{subsec:jet-volume}, we derive two lower bounds for $e$ from the positivity of $J^1(H)$, one from the log-concavity of mixed intersections on $\PP(J^1(H)^*)$ and the other from a volume estimate on the same projective bundle.
In Subsection~\ref{subsec:geometric-bounds}, we use two more ingredients of a different nature, namely apparent double points of a general projection and a canonical hyperplane section.
In the last two subsections, we combine these with the estimates for the tangent and dual varieties in Sections~\ref{sec:tangentdegree} and~\ref{sec:dual}.
Eliminating $N$, we obtain uniform bounds for the Hodge numbers, which improve
\begin{equation}\label{eq:KW-old}
 -36d-80
 \le h^{1,1}(X)-h^{2,1}(X)
 \le6d+40
\end{equation}
of \cite{KW}, and keeping $N$ fixed, we obtain bounds for the degree, the tangent degree and the degree of the dual hypersurface.


\subsection{Jet-bundle positivity and mixed intersections}\label{subsec:jet-mixed}

Let us first compare the positivity of the relevant vector bundles.
The inequalities in \cite[Corollary~3.4]{KW} were obtained from the global generation of $\Omega_X^1(2H)$, and the first jet bundle provides the corresponding inequalities with smaller coefficients.
Since the computation uses only $c_1(X)=0$, it works under the weaker assumption $K_X=0$.

\begin{proposition}\label{prop:numerical}
Let $X$ be a smooth projective threefold with $K_X=0$ and $H$ a very ample divisor on $X$.
Then we have
\begin{align}
 10H^2-c_2(X)&\ge0,\label{eq:c2bound}\\
 4H^3+2c_2(X)H-c_3(X)&\ge0,\label{eq:c3top}\\
 20H^3+2c_2(X)H+c_3(X)&\ge0,\label{eq:c21}\\
 20H^3-6c_2(X)H-c_3(X)&\ge0,\label{eq:c3seg}
\end{align}
and $c_2(X)\cdot D\ge0$ for every nef divisor class $D$.
In particular, $0\le c\le10d$ and $-20d-2c\le e\le\min\{4d+2c,\,20d-6c\}$.
\end{proposition}

\begin{proof}
Since $J^1(H)$ is globally generated, the Schur classes of $J^1(H)$ associated with $c_1^2-c_2$, $c_3$, $c_1c_2-c_3$ and $c_1^3-2c_1c_2+c_3$ are numerically nonnegative, and by \eqref{eq:jetchern} these are nothing but the four inequalities above.
The last inequality is Miyaoka's semipositivity theorem for minimal threefolds with numerically trivial canonical class (see \cite[Section~2]{KW}).
\end{proof}

Here an inequality between classes of codimension $2$ means that the difference is nonnegative after intersecting with every nef divisor, and for classes of codimension $3$ we use the degree convention of Subsection~\ref{subsec:notation}.
Note that \eqref{eq:c2bound} and \eqref{eq:c3seg} are the specializations of \cite[Corollary~2.4]{LZ} to the case $c_1(X)=0$, and they improve the corresponding inequalities obtained from $\Omega_X^1(2H)$.

In the rest of this subsection and the next, we write
\[
 J=J^1(H),\qquad p:\mathcal T=\PP(J^*)\to X,\qquad
 \xi=c_1(\cO_{\mathcal T}(1)),\qquad \eta=p^*H.
\]
Then $\mathcal T$ is a smooth sixfold, $\xi$ is nef since $J$ is globally generated, and $\eta$ is nef since $H$ is ample.
Since $\operatorname{rk}J=4$ and $\PP(E)$ parametrizes lines in our convention, the projective-bundle formula reads
\begin{equation}\label{eq:projective-bundle-pushforward}
 p_*(\xi^{3+i})=s_i(J^*)\qquad(i\ge0),
\end{equation}
and hence, by \eqref{eq:jetchern}, we have
\begin{equation}\label{eq:jet-mixed-intersections}
 \int_{\mathcal T}\xi^4\eta^2=4d,\qquad
 \int_{\mathcal T}\xi^5\eta=10d-c,\qquad
 \int_{\mathcal T}\xi^6=20d-6c-e.
\end{equation}

\begin{theorem}\label{thm:logconcave}
Let $X$ be a smooth projective threefold with $K_X=0$ and $H$ a very ample divisor on $X$.
Then we have
\begin{equation}\label{eq:logconcave}
 e\ge-5d-c-\frac{c^2}{4d},
\end{equation}
and the equality is attained by the quintic threefold $X_5\subset\PP^4$.
This bound is at least as strong as $e\ge-20d-2c$, and strictly stronger if $c<10d$.
\end{theorem}

\begin{proof}
Since $\xi$ and $\eta$ are nef, the Khovanskii--Teissier inequality for mixed intersections (see \cite[Section~1.6]{Laz}) shows
\[
 \left(\int_{\mathcal T}\xi^5\eta\right)^2
 \ge
 \left(\int_{\mathcal T}\xi^6\right)
 \left(\int_{\mathcal T}\xi^4\eta^2\right).
\]
Here we do not need $\xi$ to be big.
In fact, the inequality follows from the Hodge index theorem on a general complete intersection surface for an ample perturbation of $\xi$, and then we may pass to the nef limit.
Substituting \eqref{eq:jet-mixed-intersections}, we obtain $(10d-c)^2\ge4d(20d-6c-e)$, which is \eqref{eq:logconcave} since $d>0$.
For the quintic threefold we have $(d,c,e)=(5,50,-200)$, and the equality holds.
Finally, $0\le c\le10d$ by Proposition~\ref{prop:numerical}, and
\[
 \left(-5d-c-\frac{c^2}{4d}\right)-(-20d-2c)
 =\frac{(10d-c)(6d+c)}{4d}\ge0.
\]
\end{proof}

We will next obtain a complementary lower bound by a volume argument on the same projective bundle.


\subsection{A volume inequality for the first jet bundle}\label{subsec:jet-volume}

In this subsection, $X$ is a smooth projective threefold with $K_X=0$ and $H$ is a very ample divisor on $X$.
Sun \cite[Corollary~6.3]{Sun} proved the inequality
\begin{equation}\label{eq:Sun}
 e+10c\ge0.
\end{equation}
We will improve the coefficient $10$ to $6$ by using the first jet bundle, and in fact obtain a nonlinear refinement.
The idea is to combine a section estimate for the symmetric powers of $J$ with the algebraic Morse inequality on the projective bundle $\mathcal T$ of the previous subsection.
Recall that $0\le c\le10d$ by Proposition~\ref{prop:numerical}.

\begin{theorem}\label{thm:jet-volume}
In the above setting, we have
\begin{equation}\label{eq:jet-volume-Chern}
 e\ge40d\Big[\left(1-\frac{c}{10d}\right)^{3/2}-1\Big].
\end{equation}
In particular, $e+6c\ge0$. Moreover, the following conditions are equivalent:
\begin{enumerate}[label=\textup{(\roman*)},itemsep=1pt,topsep=3pt]
 \item $e+6c=0$;
 \item $c=0$;
 \item there exists a finite \'{e}tale surjective morphism $\varpi:A\to X$ from an abelian threefold $A$.
\end{enumerate}
The equality in \eqref{eq:jet-volume-Chern} is attained both by abelian threefolds and by quintic threefolds.
\end{theorem}

For the proof and later comparisons, we put
\begin{equation}\label{eq:jet-volume-correction}
 s:=\sqrt{1-\frac{c}{10d}},\qquad
 \mathcal Q(d,c):=20d(1-s)^2(1+2s),
\end{equation}
and prove the equivalent form\footnote{The term $\mathcal Q(d,c)$ appears as the result of an optimization in the proof below.
The two forms are equivalent since $c=10d(1-s^2)$ and hence $\mathcal Q(d,c)-6c=40d\bigl[(1-\frac{c}{10d})^{3/2}-1\bigr]$.}
\begin{equation}\label{eq:jet-volume-gap}
 e+6c\ge\mathcal Q(d,c)\ge0.
\end{equation}

Let us first recall the notion of semistability we need.
For a torsion-free sheaf $F$ of rank $\rho>0$ on $X$, we define $\mu_H(F):=c_1(F)\cdot H^2/\rho$ and $\mu_H^+(F):=\max_{0\ne G\subseteq F}\mu_H(G)$, the latter being the slope of the first nonzero term of the Harder--Narasimhan filtration.
The sheaf $F$ is called $\mu_H$-semistable (resp. stable) if $\mu_H(G)\le\mu_H(F)$ (resp. $<$) for every subsheaf $G$ with $0<\operatorname{rk}G<\rho$, and a vector bundle is called $\mu_H$-polystable if it is a direct sum of stable bundles of the same slope.
Note that $\mu_H^+(F)=\mu_H(F)$ if $F$ is semistable.

We keep the notation $J$, $p:\mathcal T\to X$, $\xi$ and $\eta$ of Subsection~\ref{subsec:jet-mixed}.
Since $\mathcal T=\PP(J^*)$ parametrizes $1$-dimensional quotients of $J$ in our convention, we have $p_*\cO_{\mathcal T}(m)=\operatorname{Sym}^mJ$ for $m\ge0$.

\begin{lemma}\label{lem:jet-semistability}
For every integer $m\ge0$, the bundle $\operatorname{Sym}^mJ$ is $\mu_H$-semistable of slope $md$.
\end{lemma}

\begin{proof}
By Yau's theorem \cite{Yau} and the Kobayashi--Hitchin correspondence, $\Omega_X^1$ and its symmetric powers are $\mu_H$-polystable of slope $0$ (see for example \cite{Kobayashi}).
Twisting the jet sequence by $-H$, we obtain
\[
 0\longrightarrow\Omega_X^1\longrightarrow J(-H)
 \longrightarrow\cO_X\longrightarrow0,
\]
and the induced filtration on $\operatorname{Sym}^m(J(-H))$ has the successive quotients $\operatorname{Sym}^j\Omega_X^1$ for $0\le j\le m$.
Indeed, this can be seen locally by splitting the sequence, and the filtration does not depend on the splitting.
Since each quotient is semistable of slope $0$, so is $\operatorname{Sym}^m(J(-H))$, and hence $\operatorname{Sym}^mJ$ is semistable of slope $md$.
\end{proof}

We next recall Langer's section estimate \cite[Theorem~3.3]{Langer} in the explicit form given in \cite[Theorem~2.5]{Sun}: for a torsion-free sheaf $F$ of rank $\rho$ such that $\mu_H^+(F)\ge0$, we have
\begin{equation}\label{eq:Langer-sections}
 h^0(X,F)\le
 \rho d\binom{\mu_H^+(F)/d+3+f(\rho)}{3},
 \qquad f(\rho)=-1+\sum_{j=1}^{\rho}\frac1j,
\end{equation}
where the binomial coefficient is understood as the polynomial $\binom{u}{3}=u(u-1)(u-2)/6$.
The dependence on the rank $\rho$ will be important below, and we note that $0\le f(\rho)\le\log\rho$.

Recall also that the volume of a Cartier divisor $D$ on a smooth projective variety $Z$ of dimension $n$ is defined by
\[
 \operatorname{vol}_Z(D) :=\limsup_{m\to\infty}\frac{n!}{m^n}h^0\bigl(Z,\cO_Z(mD)\bigr),
\]
and that of a rational divisor is defined by the same formula over sufficiently divisible $m$.
The volume depends only on the numerical class of $D$, and $\operatorname{vol}_Z(D)=D^n$ if $D$ is nef \cite[Section~2.2]{Laz}.

\begin{proposition}\label{prop:volume-upper}
For every rational number $t$ with $0\le t\le1$, we have
\begin{equation}\label{eq:volume-upper}
 \operatorname{vol}_{\mathcal T}(\xi-t\eta)\le20(1-t)^3d.
\end{equation}
\end{proposition}

\begin{proof}
Fix $t\in\mathbb Q\cap[0,1]$ and let $m$ run over the positive integers such that $mt\in\mathbb Z$.
We put
\[
 \mathcal L_{m,t}=\cO_{\mathcal T}(m)\otimes p^*\cO_X(-mtH),\qquad
 F_{m,t}=\operatorname{Sym}^mJ\otimes\cO_X(-mtH),
\]
so that $c_1(\mathcal L_{m,t})=m(\xi-t\eta)$ and $H^0(\mathcal T,\mathcal L_{m,t})\cong H^0(X,F_{m,t})$ by the projection formula.
By Lemma~\ref{lem:jet-semistability}, $F_{m,t}$ is semistable with
\[
 \rho_m=\operatorname{rk}F_{m,t}=\binom{m+3}{3}
       =\frac{m^3}{6}+O(m^2),\qquad
 \mu_H^+(F_{m,t})=(1-t)md,
\]
and hence, by \eqref{eq:Langer-sections} and $f(\rho_m)=O(\log m)$, we have
\[
 h^0(\mathcal T,\mathcal L_{m,t})
 \le \rho_m d\binom{(1-t)m+3+f(\rho_m)}{3} =\frac{d(1-t)^3}{36}m^6+O_t(m^5\log m),
\]
where the constant in the error term may depend on $t$ and $(X,H)$ but not on $m$.
Multiplying by $6!/m^6$ and taking the limsup, we obtain \eqref{eq:volume-upper}.
Note that we need neither the nefness of $\xi-t\eta$ nor the uniformity of the error term in $t$.
The same estimate applies to the endpoint $t=1$, where $\mu_H^+(F_{m,1})=0$, and in this case we have $h^0(\mathcal T,\mathcal L_{m,1})=O(m^3\log^3m)$, so that $\operatorname{vol}_{\mathcal T}(\xi-\eta)=0$.
\end{proof}

At $t=0$, since $\xi$ is nef, we have
\[
 20d-6c-e=\int_{\mathcal T}\xi^6=\operatorname{vol}_{\mathcal T}(\xi)\le20d,
\]
and hence $e+6c\ge0$.
For the nonlinear refinement, we use the algebraic Morse inequality \cite[Proposition~3.3]{BFJ}: for nef rational divisor classes $A$ and $B$ on an $n$-dimensional smooth projective variety $Z$, we have
\[
 \operatorname{vol}_Z(A-B)\ge A^n-nA^{n-1}B,
\]
where $A-B$ is not necessarily nef or big.

\begin{proposition}\label{prop:jet-volume-parameter}
For every real number $t\in[0,1]$, we have
\begin{equation}\label{eq:jet-volume-parameter}
 e+6c\ge6tc-60t^2d+20t^3d.
\end{equation}
\end{proposition}

\begin{proof}
For a rational number $0\le t\le1$, the Morse inequality for $\xi$ and $t\eta$ reads
\[
 \operatorname{vol}_{\mathcal T}(\xi-t\eta)
\ge \int_{\mathcal T} (\xi^6-6t\xi^5\eta),
\]
and hence, by \eqref{eq:jet-mixed-intersections} and Proposition~\ref{prop:volume-upper}, we have
\[
 20d-(e+6c)-60td+6tc
 \le20d-60td+60t^2d-20t^3d.
\]
This is \eqref{eq:jet-volume-parameter} for rational $t$, including the endpoint $t=1$, where we do not need the nefness of $\xi-t\eta$.
The inequality for real $t$ follows by continuity.
\end{proof}

\begin{proof}[Proof of Theorem~\ref{thm:jet-volume}]
We optimize \eqref{eq:jet-volume-parameter} in $t$.
Put $F(t)=6ct-60dt^2+20dt^3$.
Since $F''(t)=120d(t-1)\le0$ on $[0,1]$, the function $F$ is concave there, and since $F'(t)=60d(t^2-2t+c/(10d))$, its maximum is attained at $t_*=1-s$, where $s$ is as in \eqref{eq:jet-volume-correction}.
Then we have
\[
 \max_{0\le t\le1}F(t)=20d(1-s)^2(1+2s)=\mathcal Q(d,c),
\]
which proves \eqref{eq:jet-volume-gap}.
Here the endpoint $c=10d$, where $t_*=1$, is covered by Proposition~\ref{prop:jet-volume-parameter}.
Using $c=10d(1-s^2)$, we may rewrite $-6c+\mathcal Q(d,c)=40d(s^3-1)$, which is the equivalent form \eqref{eq:jet-volume-Chern}.

We next consider the equality case.
If $c>0$, then $s<1$ and $\mathcal Q(d,c)>0$, so that $e+6c>0$.
Suppose that $c=0$.
By Yau's theorem \cite{Yau}, there exists a Ricci-flat K\"ahler metric $\kappa$ in the K\"ahler class $H$, and by the Chern--Weil formula we have
\[
 c=\int_X c_2(X)\wedge \kappa=C \int_X |\mathrm{Rm}(\kappa)|^2 \kappa^3
\]
for some constant $C>0$.
Hence $c=0$ means that $\kappa$ is flat, and then $X$ admits a finite \'{e}tale cover $\varpi:A\to X$ by a complex torus by the Bieberbach theorem \cite{Wolf}.
Since $\varpi^*H$ is ample, $A$ is an abelian threefold, and $0=e(A)=(\deg\varpi)e(X)$ implies $e=0$.
Conversely, if there is a finite \'{e}tale cover $\varpi:A\to X$ by an abelian threefold, then $\varpi^*T_X\cong T_A$ is trivial, and since $\varpi^*$ is injective on rational cohomology, the rational Chern classes of $T_X$ vanish, in particular $c=e=0$.
This proves the equivalence of the three conditions.

Finally, an abelian threefold has $c=e=0$, and the equality holds in \eqref{eq:jet-volume-Chern}.
For a quintic threefold, we have $(d,c,e)=(5,50,-200)$ and $s=0$, and both sides of \eqref{eq:jet-volume-Chern} are $-200$.
\end{proof}

\begin{remark}\label{rem:type-A-equality}
If moreover $H^1(X,\cO_X)=0$, then the equality case $e+6c=0$ in Theorem~\ref{thm:jet-volume} is exactly the type A case in the classification of \cite{OS, HK}.
\end{remark}


\subsection{Secant and hyperplane-section estimates}\label{subsec:geometric-bounds}

We next derive a lower bound for $e$ from the apparent double points of a general projection, which becomes strict in the stable range by a degree estimate for the secant variety.

\begin{proposition}\label{prop:double-point}
Let $(X,H)$ be a very amply polarized Calabi--Yau threefold embedded by $|H|$ in $\PP^N$.
Then we have
\begin{equation}\label{eq:double-point-lower}
 e\ge-d^2+35d-7c,
\end{equation}
and, if $N\ge7$,
\begin{equation}\label{eq:secant-lower}
 e\ge-d^2+35d-7c+(N-5)(N-6)+2.
\end{equation}
\end{proposition}

\begin{proof}
Suppose first that $N\ge7$, and let $f:X\to\PP^6$ be a general linear projection.
Since the centre of the projection is disjoint from $\Tan(X)$, the differential of $f$ is injective, and by the generic projection theorem the image $f(X)$ has only isolated double points, at each of which two smooth branches meet transversely.
Let $\sigma$ be the number of these double points, or equivalently the number of unordered pairs of distinct points of $X$ with the same image.
If $N\le6$, we compose the embedding with a linear inclusion into $\PP^6$ and put $\sigma=0$.
In either case, let $N_f$ be the normal bundle defined by $0\rightarrow T_X\xrightarrow{df}f^*T_{\PP^6}  \rightarrow N_f \rightarrow0$.
Then the double-point formula \cite[Theorem~2.3]{CO} reads
\[
 d^2=\int_X c_3(N_f)+2\sigma,
\]
where each double point is counted twice since its two preimages can be ordered in two ways.
Since $f^*\cO_{\PP^6}(1)\cong\cO_X(H)$, we have
\[
 c(N_f)=\frac{(1+H)^7}{c(T_X)},\qquad
 c_3(N_f)=35H^3-7c_2(X)\cdot H-c_3(X),
\]
and hence $2\sigma=d^2-35d+7c+e$.
Since $\sigma\ge0$, we obtain \eqref{eq:double-point-lower}.

Suppose that $N\ge7$.
Then $\dim\Sec(X)=7$ by Theorem~\ref{thm:stable}, and $\sigma=\mu\deg\Sec(X)$, where $\mu$ is the number of secant lines through a general point of $\Sec(X)$.
If $N\ge8$, the minimal secant-degree bound \cite[Theorem~4.2]{CR} shows
\[
 \sigma\ge\deg\Sec(X)\ge\binom{N-5}{2},
\]
and the same inequality holds for $N=7$ as well, since $\Sec(X)=\PP^7$ and $\mu\ge1$.
If the equality holds, then $X$ has the minimal number of apparent double points, which is one for $N=7$, and $X$ would be rational by \cite[Corollary~4.5(iii)]{CR}.
This contradicts $h^{3,0}(X)=1$.
Therefore $\sigma\ge\binom{N-5}{2}+1$, which proves \eqref{eq:secant-lower}.
\end{proof}

The first bound \eqref{eq:double-point-lower} is attained if and only if $N\le6$.
In terms of the dual degree, the above proof shows $\deg X^\vee=d^2-31d+9c-2\sigma$, and hence the corresponding upper bound for $\deg X^\vee$ is also attained for $N\le6$.
For $N\ge7$, the equality in \eqref{eq:secant-lower} is equivalent to $\sigma=\binom{N-5}{2}+1$, and we do not know whether or not this value is attained by a Calabi--Yau threefold.

Let us next consider a smooth hyperplane section $i:S\hookrightarrow X$, and put $p_g(S)=h^0(S,\cO_S(K_S))$.
By the adjunction formula and the exact sequence $0\to\cO_X\to\cO_X(H)\to i_*\cO_S(K_S)\to0$ with $H^1(X,\cO_X)=0$, the canonical embedding of $S$ is nothing but the hyperplane section of $X\subset\PP^N$, and in particular $p_g(S)=N$.
We also have $h^1(S,\cO_S)=0$ by the weak Lefschetz theorem and $\int_Sc_2(S)=d+c$ by the tangent sequence.
Then the Castelnuovo inequality $K_S^2\ge3p_g(S)-7$ and \eqref{eq:RR} imply
\begin{equation}\label{eq:Castelnuovo}
 c\le2d+40,
\end{equation}
which is \cite[Proposition~2.1(2)]{KW}.
Since the right-hand side of \eqref{eq:logconcave} is decreasing in $c\ge0$, substituting \eqref{eq:Castelnuovo} into \eqref{eq:logconcave} and \eqref{eq:double-point-lower}, we obtain
\begin{equation}\label{eq:lower-logconcave-Castelnuovo}
 e\ge-8d-80-\frac{400}{d}
\end{equation}
and
\begin{equation}\label{eq:lower-double-point-Castelnuovo}
 e\ge-d^2+21d-280,
\end{equation}
respectively.

In order to bound $h^{1,1}(X)-h^{2,1}(X)$ from above, we next bound $h^{1,1}(X)$ by means of a hyperplane section.
The point is that the weak Lefschetz inequality $h^{1,1}(X)\le h^{1,1}(S)$ is strict because of a vanishing cycle.

\begin{proposition}\label{prop:lefschetz-Hodge}
Let $(X,H)$ be a very amply polarized Calabi--Yau threefold embedded by $|H|$ in $\PP^N$.
Then we have
\begin{equation}\label{eq:lefschetz-h11}
 h^{1,1}(X)\le10(N+1)-d-1=\frac{4d+5c}{6}-1,
\end{equation}
and hence
\begin{equation}\label{eq:lefschetz-Hodge}
 h^{1,1}(X)-h^{2,1}(X)
 \le\left\lfloor\frac{7d+97}{3}\right\rfloor.
\end{equation}
\end{proposition}

\begin{proof}
Let $i:S\hookrightarrow X$ be a smooth hyperplane section as above.
By the weak Lefschetz theorem, $i^*:H^2(X,\mathbb R)\to H^2(S,\mathbb R)$ is injective and compatible with the Hodge decomposition, and its image $W$ is of type $(1,1)$ since $h^{2,0}(X)=0$.
Since $W$ contains the ample class $H|_S$, the intersection form on $W$ has signature $(1,h^{1,1}(X)-1)$ by the Hodge index theorem, and since that on $H^2(S,\mathbb R)$ has signature $(1+2p_g(S),h^{1,1}(S)-1)$, its restriction to $W^\perp$ has signature
\[
 \bigl(2p_g(S),\,h^{1,1}(S)-h^{1,1}(X)\bigr).
\]
On the other hand, since $X^\vee$ is a hypersurface by Theorem~\ref{thm:dual-hypersurface}, a general pencil in $|H|$ is a Lefschetz pencil with singular members, each of which has one ordinary double point.
The vanishing sphere of such a member defines, after parallel transport to $S$ and Poincar\'e duality, a class $\delta\in H^2(S,\mathbb R)$ with $\delta^2=-2$ (see \cite[Chapters~2--3]{Vois}), and $\delta\in W^\perp$ since the sphere bounds a $3$-chain in $X$ and hence its pushforward to $H_2(X,\mathbb R)$ vanishes.
Therefore the negative index of $W^\perp$ is at least $1$, that is, $h^{1,1}(X)\le h^{1,1}(S)-1$.

By the invariants of $S$ computed above and $2d+c=12(N+1)$ from \eqref{eq:RR}, we have
\[
 h^{1,1}(S)=d+c-2-2N=10(N+1)-d,
\]
which proves \eqref{eq:lefschetz-h11}.
Since $h^{2,1}(X)\ge0$ and $c\le2d+40$, we have
\[
 h^{1,1}(X)-h^{2,1}(X)\le\frac{4d+5c}{6}-1
 \le\frac{7d+97}{3},
\]
and \eqref{eq:lefschetz-Hodge} follows from integrality.
\end{proof}

We do not know whether either bound in Proposition~\ref{prop:lefschetz-Hodge} is attained.


\subsection{Uniform Hodge bounds}

Combining the hyperplane-section estimate with the tangent-variety bound, we can eliminate the embedding dimension $N$ and obtain uniform bounds for $h^{1,1}(X)-h^{2,1}(X)$.

\begin{theorem}\label{thm:Hodge}
Let $(X,H)$ be a very amply polarized Calabi--Yau threefold with the complete embedding $X\hookrightarrow\PP^N$ defined by $|H|$.
Then we have
\begin{equation}\label{eq:uniform-Hodge-linear}
 -4d-80\le h^{1,1}(X)-h^{2,1}(X)\le\frac{173}{66}d.
\end{equation}
The equality in the lower bound holds if and only if $(X,H)$ is a quintic threefold $X_5 \subset \PP^4$, and the equality in the upper bound forces $N=23$, $d=66$ and $(h^{1,1}(X),h^{2,1}(X))=(173,0)$.
\end{theorem}

The arithmetic part of the proof of the upper bound is the following lemma.

\begin{lemma}\label{lem:integer-upper}
Let $k\ge8$, $d>0$ and $\Delta$ be integers such that $\Delta\le10k-d-1$ and $\Delta\le16d-37k+5$.
Then $66\Delta\le173d$, and the equality holds if and only if $(k,d,\Delta)=(24,66,173)$.
\end{lemma}

\begin{proof}
We put $A=10k-d-1-\Delta$ and $B=16d-37k+5-\Delta$, which are nonnegative integers by assumption.
Eliminating $k$, we have $47\Delta=123d+13-37A-10B$, and hence
\[
 47(66\Delta-173d)=858-13d-2442A-660B.
\]
The right-hand side is negative if $A\ge1$ or $B\ge2$, so that it suffices to consider the cases $A=0$ and $B=0,1$.
If $B=0$, then $17d=47k-6$ and $k\equiv7\pmod{17}$, and then $k\ge8$ implies $k\ge24$ and $d\ge66$.
If $B=1$, then $17d=47k-5$ and $k\equiv3\pmod{17}$, and hence $k\ge20$ and $d\ge55$.
In both cases the right-hand side is nonpositive, and it vanishes only when $A=B=0$, $d=66$ and $k=24$, in which case $\Delta=173$ and both hypotheses hold with equality.
\end{proof}

\begin{proof}[Proof of Theorem~\ref{thm:Hodge}]
Put $\Delta=h^{1,1}(X)-h^{2,1}(X)=e/2$.
By \eqref{eq:lower-logconcave-Castelnuovo}, we have $\Delta\ge-4d-40-200/d$.
Since $N+1\le d$ by \eqref{eq:RR} and $c\le10d$, we have $d\ge5$ and $200/d\le40$, which proves the lower bound.
It is strict for $d>5$, and if $d=5$, then $N=4$ and $X$ is a quintic threefold with $\Delta=-100=-4d-80$.

For the upper bound, suppose first that $N\ge7$ and put $k=N+1$.
By Proposition~\ref{prop:lefschetz-Hodge}, Theorem~\ref{thm:stable} and \eqref{eq:RR}, we have
\[
 \Delta\le10k-d-1,\qquad \Delta\le16d-37k+5.
\]
Adding $37$ times the first and $10$ times the second, we obtain $47\Delta\le123d+13$, and Lemma~\ref{lem:integer-upper} shows $66\Delta\le173d$.
For $N=4,5$, the complete intersections have $\Delta<0$, and for $N=6$, we have $\Delta=(49d-588-d^2)/2\le6$ and $d\ge12$ by Theorem~\ref{thm:P6exact}, so that the upper bound holds strictly in these cases.
Finally, the equality in the upper bound forces $k=24$, $d=66$ and $\Delta=173$ by Lemma~\ref{lem:integer-upper}, and then $h^{1,1}(X)\le173$ by Proposition~\ref{prop:lefschetz-Hodge}, so that $(h^{1,1}(X),h^{2,1}(X))=(173,0)$.
\end{proof}

\begin{remark}\label{rem:Hodge-comparison}
Theorem~\ref{thm:Hodge} improves \eqref{eq:KW-old}.
The above proof also gives the more precise bounds
\begin{equation}\label{eq:uniform-Hodge-nonlinear}
 \begin{gathered}
 \Delta\ge\max\left\{
       \left\lceil-4d-40-\frac{200}{d}\right\rceil,\,
       \frac{-d^2+21d-280}{2}\right\},\\
 \Delta\le\min\left\{
       \left\lfloor\frac{173d}{66}\right\rfloor,\,
       \left\lfloor\frac{123d+13}{47}\right\rfloor,\,
       \left\lfloor\frac{7d+97}{3}\right\rfloor\right\},
 \end{gathered}
\end{equation}
where $\Delta=h^{1,1}(X)-h^{2,1}(X)$.
Here the second lower bound follows from \eqref{eq:lower-double-point-Castelnuovo}, the second upper bound is $47\Delta\le123d+13$ in the above proof, which holds also for $N\le6$ since $\Delta\le6$ there, and the last upper bound is \eqref{eq:lefschetz-Hodge}.
If we keep $c$ instead of eliminating it, we also obtain the refinement
\[
 \Delta\ge\left\lceil-3c+\frac{\mathcal Q(d,c)}2\right\rceil
\]
from Theorem~\ref{thm:jet-volume}.
By Lemma~\ref{lem:integer-upper}, $173/66$ is the smallest coefficient that can be obtained from the two upper estimates and integrality alone, and the equality would require
\[
 N=23,\qquad (d,c,e)=(66,156,346),\qquad
 (h^{1,1}(X),h^{2,1}(X))=(173,0).
\]
We do not know any very amply polarized Calabi--Yau threefold with these invariants.
\end{remark}

Finally, we note that the coefficient $2$ in the surface estimate \eqref{eq:Castelnuovo} is optimal.

\begin{example}\label{ex:castelnuovo-sharp}
Let $X\subset\PP^1\times\PP^3$ be a smooth hypersurface of bidegree $(2,4)$, and let $A$ and $B$ be the restrictions of the two hyperplane classes.
Then $K_X=0$ by the adjunction formula, $H^1(X,\cO_X)=0$ by the weak Lefschetz theorem, and $c_2(X)=8AB+6B^2$.
For every integer $a\ge1$, the divisor $H_a=aA+B$ is very ample, and intersecting with $[X]=2A+4B$ in $\PP^1\times\PP^3$, we obtain
\[
\int_X H_a^3=12a+2,\qquad \int_X c_2(X)H_a=24a+44=2 \int_X H_a^3+40.
\]
Therefore the equality in \eqref{eq:Castelnuovo} is attained for arbitrarily large degree, and the coefficient $2$ cannot be decreased whatever the constant term is.
Note that this does not mean that the Hodge bounds are asymptotically optimal, since they also use other inequalities.
\end{example}


\subsection{Degree bounds and dual varieties}

The above estimates also constrain the degree for a fixed ambient dimension $N$.

\begin{lemma}\label{lem:degree-window}
Let $(X,H)$ be a very amply polarized Calabi--Yau threefold embedded by $|H|$ in $\PP^N$.
Then we have
\begin{equation}\label{eq:degree-window}
 3N-7\le d\le6(N+1).
\end{equation}
\end{lemma}

\begin{proof}
The lower bound follows from \eqref{eq:Castelnuovo} and \eqref{eq:RR}, and the upper bound from $c\ge0$.
\end{proof}

In the critical case $N=6$, the window \eqref{eq:degree-window} reads $11\le d\le42$, and the following theorem sharpens both ends.

\begin{theorem}\label{thm:P6window}
For a nondegenerate Calabi--Yau threefold $X\subset\PP^6$, we have
\[
 12\le d\le34,\qquad
 24\le\tdeg(X)\le662,\qquad
 \deg X^\vee\le312,
\]
and the equality $\deg X^\vee=312$ holds if and only if $X$ is a complete intersection $X_{2,2,3}$.
\end{theorem}

\begin{proof}
By Theorem~\ref{thm:P6exact}, we have $d\ge12$ and
\[
 e+6c=(49d-588-d^2)+6(84-2d)=-d^2+37d-84,
\]
which is nonnegative by Theorem~\ref{thm:jet-volume}.
Hence $d\le\frac{37+\sqrt{1033}}{2}<35$, that is, $d\le34$.
Since $\tdeg(X)=d^2-17d+84$ is increasing for $d\ge12$, it takes the values from $24$ at $d=12$ to $662$ at $d=34$.
As for the dual degree, Section~\ref{sec:dual} shows
\[
 \deg X^\vee=312+(d-12)(d-37)\le312,
\]
where the equality holds if and only if $d=12$ in the above range, i.e. $X$ is $X_{2,2,3}$ by Theorem~\ref{thm:P6exact}.
\end{proof}

The bound $d\le34$ improves $d\le41$ in \cite[Corollary~2.2]{KK}.
Note that the proof uses only the linear consequence $e+6c\ge0$ of the volume estimate, and no assumption on projective normality or on the conormal bundle is needed.
We do not claim the existence of a Calabi--Yau threefold of degree $34$ in $\PP^6$.

The above estimates also provide a lower bound for the degree of the dual hypersurface which is independent of the polarization.

\begin{theorem}\label{thm:dual-lower}
For every very amply polarized Calabi--Yau threefold $(X,H)$, we have
\[
 \deg X^\vee\ge78.
\]
\end{theorem}

\begin{proof}
Put $\delta^\vee=\deg X^\vee$ and $k=N+1$, so that $\delta^\vee=24k-e$ by \eqref{eq:dual-degree} and \eqref{eq:RR}.
If $N\ge7$, then $e\le20k-2d-2$ by Proposition~\ref{prop:lefschetz-Hodge} and $e\le32d-74k+10$ by \eqref{eq:stable-chern} and \eqref{eq:RR}, that is,
\[
 \delta^\vee\ge2d+4k+2,\qquad
 \delta^\vee\ge98k-32d-10,
\]
and eliminating $d$ by adding $16$ times the first inequality to the second, we obtain $17\delta^\vee\ge162k+22\ge1318$.
Since $\delta^\vee$ is even, this implies $\delta^\vee\ge78$.
If $N=6$, then by the exact formula in Section~\ref{sec:dual} and the integrality of $d$ we have
\[
 \delta^\vee=d^2-49d+756=(d-24)(d-25)+156\ge156.
\]
Finally, for $N=4,5$, the dual degrees of $X_5$, $X_{2,4}$ and $X_{3,3}$ are $320$, $320$ and $288$, respectively.
\end{proof}

\begin{remark}\label{rem:dual-equality}
Suppose that $\deg X^\vee=78$.
Then the above proof forces $N=7$, and the two lower bounds read $78\ge2d+34$ and $78\ge774-32d$, so that $d=22$, and then $c=52$ and $e=114$ by \eqref{eq:RR} and \eqref{eq:dual-degree}.
Moreover, $h^{1,1}(X)\le57$ by \eqref{eq:lefschetz-h11} while $h^{1,1}(X)-h^{2,1}(X)=e/2=57$, and hence $(h^{1,1}(X),h^{2,1}(X))=(57,0)$.
Finally, the embedding is not $2$-normal, since
\[
 h^0(X,\cO_X(2H))=38>36=h^0(\PP^7,\cO_{\PP^7}(2))
\]
by Riemann--Roch and Kodaira vanishing.
In summary, the equality $\deg X^\vee=78$ would force
\[
 N=7,\qquad (d,c,e)=(22,52,114),\qquad
 (h^{1,1}(X),h^{2,1}(X))=(57,0),
\]
but we do not know whether or not such a threefold exists.
\end{remark}

\begin{example}\label{ex:Tonoli-dual}
Let $X\subset\PP^6$ be one of Tonoli's Calabi--Yau threefolds of degree $17$ \cite{Tonoli}.
Then $(d,c,e)=(17,50,-44)$ by Theorem~\ref{thm:P6exact}, and hence $\deg X^\vee=212$ by \eqref{eq:dual-degree}.
As far as we know, this is the smallest dual degree of a very amply polarized Calabi--Yau threefold.
\end{example}


\section{Quadratic equations and conormal positivity}\label{sec:quadratic}

In this section, we consider the special case when the homogeneous ideal of $X\subset\PP^N$ is generated by quadrics. 
This assumption has two consequences. 
First, a general canonical hyperplane section of $X$ is cut out by quadrics, and then the surface estimate \eqref{eq:Castelnuovo} can be improved.
Second, the twisted conormal bundle $N_{X/\PP^N}^*(2H)$ is globally generated, and then the log-concavity argument in the proof of Theorem~\ref{thm:logconcave} applies to it as well.
In fact, these two inequalities determine all the invariants for $N=7$, and leave only two possible values of $e$ for $N=8$.

\subsection{Quadratically defined embeddings}

We say that the embedding $X\subset\PP^N$ is quadratically defined if its homogeneous ideal is generated by quadrics.
Recall that the surface estimate \eqref{eq:Castelnuovo} was obtained from the Castelnuovo inequality for a general canonical surface $S\in|H|$.
If moreover the canonical model of $S$ is an intersection of quadrics, we may use a stronger inequality of \cite{KW}, as follows.

\begin{proposition}\label{prop:quadratic-refinement}
Let $(X,H)$ be a very amply polarized Calabi--Yau threefold embedded by $|H|$ in $\PP^N$, and assume that a general smooth hyperplane section $S\in|H|$, in its canonical embedding in $\PP^{N-1}$, is an irreducible component of the common zero locus of the quadrics containing it.
This is the case, for example, if the embedding $X\subset\PP^N$ is quadratically defined.
Then we have
\begin{equation}\label{eq:quadratic-degree}
 c\le d+48,\qquad d\ge4N-12.
\end{equation}
\end{proposition}

\begin{proof}
As we saw in Subsection~\ref{subsec:geometric-bounds}, the canonical embedding of $S$ is nothing but the hyperplane section of $X\subset\PP^N$, and hence the assumption on $S$ allows us to apply \cite[Proposition~2.2(3)]{KW}, which shows $c\le d+48$.
If the ideal of $X$ is generated by quadrics, then $S$ is scheme-theoretically cut out by quadrics in $\PP^{N-1}$, so that the assumption is satisfied.
The second inequality follows from the first by substituting $c=12(N+1)-2d$ (see \eqref{eq:RR}).
\end{proof}

Recall that the embedding $X\subset\PP^N$ is said to satisfy property $N_2$ if it is projectively normal, the homogeneous ideal is generated by quadrics, and the first syzygies among the quadratic generators are linear.

\begin{corollary}\label{cor:quadratic-tau}
Under the hypotheses of Proposition~\ref{prop:quadratic-refinement}, assume that $N\ge7$ and that the embedding satisfies property $N_2$.
Then $\tdeg(X)=1$.
\end{corollary}

\begin{proof}
Since $X$ is not secant defective by Theorem~\ref{thm:stable}, the assertion follows from the syzygy criterion of \cite[Example~29]{HGR}.
\end{proof}

\begin{remark}
The degree bound \eqref{eq:quadratic-degree} is sharp.
In fact, the complete intersection $X_{2,2,2,2}\subset\PP^7$ has $d=2^4=4\cdot7-12$.
Note that quadratic generation is enough for the degree estimate, and property $N_2$ is used only in Corollary~\ref{cor:quadratic-tau}.
\end{remark}


\subsection{Conormal positivity and quadratic generators}

We next consider the twisted conormal bundle $E=N_{X/\PP^N}^*(2H)$.
If the embedding is quadratically defined, then $E$ is globally generated, and its first Chern class is smaller than that of the jet bundle when $N$ is small.
This allows us to apply the mixed-intersection log-concavity once more, now to $E$ instead of $J^1(H)$.

\begin{proposition}\label{prop:conormal}
Let $(X,H)$ be a very amply polarized Calabi--Yau threefold embedded by $|H|$ in $\PP^N$, and suppose that its homogeneous ideal is generated by quadrics.
Then $N\ge7$, and if $N=7$, then $X$ is a complete intersection of four quadrics.
If $N\ge8$, we put $k=N-7$ and
\[
 a=\left(\binom{k+1}{2}-4\right)d+c,\qquad
 b=\left(\binom{k+2}{3}-4k-8\right)d+(k+4)c+e.
\]
Then we have
\begin{equation}\label{eq:conormal-inequalities}
 a>0,\qquad 0\le b\le\frac{a^2}{kd}.
\end{equation}
\end{proposition}

\begin{proof}
Let $r=N-3$ be the rank of $E$.
Since the ideal is generated by quadrics, the evaluation map $H^0(\PP^N,\mathcal I_X(2))\otimes\cO_X\to E$ is surjective, and hence $E$ is globally generated.
By the twisted conormal sequence and the Euler sequence, we have
\[
 c(E)=\frac{c(\Omega_{\PP^N}^1(2)|_X)}{c(\Omega_X^1(2H))},
 \qquad
 c(\Omega_{\PP^N}^1(2)|_X)=\frac{(1+H)^{N+1}}{1+2H},
\]
and in particular $c_1(E)=(N-7)H$.
Since $\det E$ is nef, we get $N\ge7$.

Suppose that $N=7$.
Then $d\ge16$ by Proposition~\ref{prop:quadratic-refinement}.
On the other hand, since the base scheme of the quadrics through $X$ is $X$ itself, which has codimension $4$, four general quadrics through $X$ form a complete intersection containing $X$, whose degree is $16$ by B\'ezout, and hence $d\le16$.
Therefore $d=16$, and this complete intersection coincides with $X$ by unmixedness.

Suppose next that $N\ge8$ and put $k=N-7$.
Expanding the above formula through degree $3$, we obtain
\begin{align*}
 c_1(E)&=kH, \\
 c_2(E)&=\left(\binom{k}{2}+4\right)H^2-c_2(X),\\
 c_3(E)&=\left(\binom{k}{3}+4k-8\right)H^3
             +(4-k)c_2(X)H+c_3(X),
\end{align*}
and hence, for $k\ge1$, the Segre classes of $E^*$ satisfy
\[
 \int_Xs_1(E^*)H^2=kd,\qquad
 \int_Xs_2(E^*)H=a,\qquad
 \int_Xs_3(E^*)=b.
\]
Let $p:\PP(E^*)\to X$ be the projection, and put $\eta=p^*H$ and $\xi=c_1(\cO_{\PP(E^*)}(1))$.
Then $\xi$ is nef since $E$ is globally generated, and $\eta$ is nef since $H$ is ample.
By \eqref{eq:segre-convention} and the projective-bundle formula, we have
\[
 \int_{\PP(E^*)}\xi^r\eta^2=kd,\qquad
 \int_{\PP(E^*)}\xi^{r+1}\eta=a,\qquad
 \int_{\PP(E^*)}\xi^{r+2}=b.
\]
Therefore the same argument as in the proof of Theorem~\ref{thm:logconcave}, namely the nonnegativity of these numbers and the Khovanskii--Teissier inequality \cite[Section~1.6]{Laz}, shows $a\ge0$ and $0\le b\le a^2/(kd)$.

We will show that $a>0$.
Substituting \eqref{eq:RR}, we may write
\[
 a=\left(\binom{k+1}{2}-6\right)d+12(N+1),
\]
which is clearly positive for $N\ge10$.
For $N=8$, we have $a=108-5d$, which does not vanish since $108$ is not divisible by $5$.
For $N=9$, if $a=120-3d=0$, then $d=c=40$, and the inequality $b\le a^2/(kd)$ forces $b=e-240=0$.
This contradicts the bound $e\le20(N+1)-2d-2=118$ of Proposition~\ref{prop:lefschetz-Hodge}, and the proof is complete.
\end{proof}

We can also bound the degree by counting the quadrics through $X$.
The point is that the kernel of the evaluation map $\cO_X^{\oplus\ell}\to E$ must have rank at least $2$, where $\ell$ is the number of independent quadrics through $X$.
In the following, we write $\varepsilon:=h^1(\PP^N,\mathcal I_X(2))$, which measures the failure of $2$-normality of the embedding.

\begin{proposition}
\label{prop:quadratic-generators}
Under the hypotheses of Proposition~\ref{prop:conormal}, assume that $N\ge8$.
Then we have
\[
 4N-12\le d\le
 \min\left\{6(N+1),\,\frac{N(N-3)}2+\varepsilon\right\}.
\]
Moreover, if $d=N(N-3)/2+\varepsilon$, then $b=0$ in Proposition~\ref{prop:conormal}, and this does not happen for $N\ge12$, so that the second upper bound can be replaced by $N(N-3)/2+\varepsilon-1$ in that case.
Note that the embedding is $2$-normal if and only if $\varepsilon=0$.
\end{proposition}

\begin{proof}
Let $\ell=h^0(\PP^N,\mathcal I_X(2))$, and let $E$ be as in the proof of Proposition~\ref{prop:conormal}.
By the ideal-sheaf sequence, Riemann--Roch and Kodaira vanishing, we have
\[
 \ell=\binom{N+2}{2}-h^0(X,\cO_X(2H))+\varepsilon
      =\frac{(N+1)(N-2)}2-d+\varepsilon.
\]
Let $K$ be the kernel of the surjection $\cO_X^{\oplus\ell}\to E$, which is a vector bundle of rank
\[
 t=\ell-(N-3)=\frac{N^2-3N+4}{2}-d+\varepsilon.
\]
Then $c(K)=c(E)^{-1}$, and in particular $c_2(K)=s_2(E^*)$ and $c_3(K)=-s_3(E^*)$.
If $t\le1$, then $c_2(K)=0$, which contradicts $\int_Xc_2(K)H=a>0$ (Proposition~\ref{prop:conormal}).
Therefore $t\ge2$, which is nothing but $d\le N(N-3)/2+\varepsilon$.
The other bounds are Proposition~\ref{prop:quadratic-refinement} and $c\ge0$.

If $d=N(N-3)/2+\varepsilon$, then $t=2$, so that $c_3(K)=0$ and $b=-\int_Xc_3(K)=0$.
Suppose moreover that $N\ge12$, and put $g(k)=\binom{k+2}{3}-4k-8$.
Then $b=0$ reads $e=-g(k)d-(k+4)c$, and hence
\[
 e+6c=-g(k)d-(k-2)c.
\]
Since $k=N-7\ge5$, we have $g(k)\ge7$ and $k-2>0$, and $c\ge0$ by Proposition~\ref{prop:numerical}, so that the right-hand side is negative.
This contradicts Theorem~\ref{thm:jet-volume}, and the improvement by $1$ follows from integrality.
\end{proof}

It is worth noting that, for $2$-normal embeddings, the bound $N(N-3)/2$ is stronger than $6(N+1)$ exactly when $8\le N\le15$, and the improvement by $1$ applies for $12\le N\le15$.
We do not need any assumption on the first syzygies here.


\subsection{Embeddings in \texorpdfstring{$\PP^8$}{P8}}

Let us apply these inequalities to the case $N=8$.

\begin{corollary}
\label{cor:quadratic-P8}
Let $(X,H)$ be a Calabi--Yau threefold embedded by $|H|$ in $\PP^8$, with homogeneous ideal generated by quadrics.
Then
\[
 d=20,\qquad c=68,\qquad e\in\{-120,-118\}.
\]
If the embedding is $2$-normal, then $e=-120$.
\end{corollary}

\begin{proof}
By Proposition~\ref{prop:quadratic-refinement}, we have $d\ge20$, and by \eqref{eq:RR}, $c=108-2d$.
For $k=1$, the quantities in Proposition~\ref{prop:conormal} are $a=c-3d=108-5d$ and $b=e+5c-11d$, and the condition $a>0$ forces $d\le21$.
If $d=21$, then $c=66$ and $a=3$, so that $b=e+99$ is an odd integer with $0\le b\le9/21<1$ by \eqref{eq:conormal-inequalities}, which is absurd.
Therefore $d=20$, $c=68$ and $a=8$, and $b=e+120$ is an even integer with $0\le b\le64/20$, namely $b=0$ or $2$.
In other words, $e=-120$ or $-118$.
If the embedding is $2$-normal, then $d=20=N(N-3)/2$, and the equality case of Proposition~\ref{prop:quadratic-generators} implies $b=0$, i.e. $e=-120$.
\end{proof}

\begin{example}\label{ex:quadratic-P8}
The degree $d=20$ in Corollary~\ref{cor:quadratic-P8} is attained by the Calabi--Yau threefold
\[
 X=\Gr(2,5)\cap \Lambda\cap Q_1\cap Q_2\subset \Lambda\cong\PP^8,
\]
where $\Gr(2,5)\subset\PP^9$ is the Pl\"ucker embedding, $\Lambda\subset\PP^9$ is a general hyperplane and $Q_1,Q_2$ are general quadrics.
The ideal of $X$ in $\Lambda$ is generated by the restrictions of the Pl\"ucker quadrics and the two quadrics $Q_1,Q_2$.
Since the Pl\"ucker coordinate ring is Cohen--Macaulay and $\Lambda,Q_1,Q_2$ form a regular sequence, the embedding is projectively normal, and hence $e=-120$ by Corollary~\ref{cor:quadratic-P8}.
We also have $h^{1,1}(X)=1$ by the weak Lefschetz theorem, and hence $h^{2,1}(X)=61$.
\end{example}

\begin{remark}
Note that Corollary~\ref{cor:quadratic-P8} is a numerical restriction and not a classification.
We do not know whether or not the value $e=-118$ is realized, which would require the failure of $2$-normality.
\end{remark}


\paragraph{Acknowledgment.}
The author thanks Hao Sun for drawing his attention to \cite{Sun}, which led to improvements in the present paper. 
Thanks also go to Shinobu Hosono and Pelham Wilson for useful comments on the draft of the paper. 
This work is supported by JSPS Grants-in-Aid Kiban (C) 22K03296 and Kiban (A) 23H00083.

\paragraph{AI disclosure.}
ChatGPT was used for mathematical exploration, discussing proofs and editing the paper. 
The author verified the mathematical results and proofs and takes full responsibility for the contents.


\par\noindent{\scshape \small
Department of Mathematics \\
Faculty of Science and Engineering\\
Waseda University \\
3-4-1 Okubo, Shinjuku, Tokyo, 169-8555, Japan}
\par\noindent{\ttfamily  a\_kanazawa@waseda.jp}

\end{document}